\documentclass[11pt]{article}
\usepackage{graphicx}
\usepackage{amsmath,amsfonts,amssymb,amsthm}
\title{Reflections of planar convex bodies}
\author{Rolf Schneider}
\date{}
\newcommand{\Sn}{{\mathbb S}^1}

\newcommand{\ur}{\mbox{\boldmath$u$}}

\newtheorem{lemma}{Lemma}
\newtheorem{theorem}{Theorem}

\begin{document}
\maketitle

\begin{abstract}
It is proved that every convex body in the plane has a point such that the union of the body and its image under reflection in the point is convex. If the body is not centrally symmetric, then it has, in fact, three affinely independent points with this property.
\end{abstract}

\section{Introduction}\label{sec1}

In November 2013, Shiri Artstein--Avidan asked me the following question: `Does every convex body $K$ in the plane have a point $z$ such that the union of $K$ and its reflection in $z$ is convex?' It seemed hard to believe that such a simple question should not have been asked before, and that its answer should be unknown. However, neither a reference nor a counterexample turned up. This note gives a proof. Let us call the point $z$ a {\em convexity point} of $K$ if $K\cup (2z-K)$ is convex. We prove the following stronger result.

\begin{theorem}
A convex body in the plane which is not centrally symmetric has three affinely independent convexity points.
\end{theorem}

In the next section we collect some preparations and at the end explain the idea of the proof. The Theorem is then proved in Section \ref{sec3}.

\section{Some Preparations}\label{sec2}

We work in the Euclidean plane ${\mathbb R}^2$, with scalar product $\langle\cdot,\cdot\rangle$ and unit circle ${\mathbb S}^1$. By $[x,y]$ we denote the closed segment with endpoints $x$ and $y$. The set of convex bodies (nonempty, compact, convex subsets) in ${\mathbb R}^2$ is denoted by ${\mathcal K}^2$. 

\begin{lemma}\label{L1}
For $K,L\in{\mathcal K}^2$, the set $K\cup L$ is convex if and only if
\begin{equation}\label{2.1}
{\rm bd \,conv}\,(K\cup L)\subset K\cup L.
\end{equation}
\end{lemma}

\begin{proof}Suppose that (\ref{2.1}) holds. Let $a,b\in K\cup L$ and $c \in [a,b]$; then $c\in {\rm conv}(K\cup L)$. Because of (\ref{2.1}), $K$ and $L$ cannot be strongly separated by a line, hence $K\cap L\not=\emptyset$. Let $p\in K\cap L$. If $p=c$, then $c\in K\cup L$. If $p\not=c$, the ray $\{p+\lambda(c-p):\lambda\ge 0\}$ meets ${\rm bd \; conv}(K\cup L)$ in a point $q$ such that $c\in[p,q]$. Then $q\in K\cup L$ by (\ref{2.1}), hence $c\in K\cup L$. Thus, $K\cup L$ is convex. If, conversely, $K\cup L$ is convex, then (\ref{2.1}) holds trivially. 
\end{proof}

For $K\in{\mathcal K}^2$ and $u\in{\mathbb S}^1$, let $H(K,u)$ be the supporting line of $K$ with outer unit normal vector $u$. The line
$$ M_K(u):= \frac{1}{2}[H(K,u)+H(K,-u)]$$
is called the {\em middle line} of $K$ with normal vector $u$. By $F(K,u):=K \cap H(K,u)$ we denote the face of $K$ with outer normal vector $u$. The convex set (point or segment)
$$ Z_K(u):= \frac{1}{2}[F(K,u)+F(K,-u)]$$
is called the {\em middle set} of $K$ with normal vector $u$. Thus, $M_K(u)=M_K(-u)$, $Z_K(u)=Z_K(-u)$, and $Z_K(u) \subset M_K(u)$.

By an {\em edge} of $K\in{\mathcal K}^2$ we mean a one-dimensional face of $K$.

\begin{lemma}\label{L2}
Let $K\in{\mathcal K}^2$ and suppose that the boundary of $K$ does not contain two parallel edges. Let $z\in{\mathbb R}^2$. Then $z$ is a convexity point of $K$ if and only if 
\begin{equation}\label{2.2}
\forall\, u\in{\mathbb S}^1:\;z\in M_K(u)\; \Rightarrow\; z\in Z_K(u).
\end{equation}
\end{lemma}

\begin{proof} We write $2z-K=:L$. Suppose, first, that $z$ is a convexity point of $K$ , thus $K\cup L$ is convex. Let $u\in{\mathbb S}^1$ be such that $z\in M_K(u)$. Then $H(K,u)$ and $H(K,-u)$ are common support lines of $K$ and $L$. Choose $x\in F(K,u)$ and $x'\in F(K,-u)$. Then  $y:= 2z-x'\in  F(L,u)$. We have $[x,y]\subset{\rm bd\,conv}(K\cup L)$ and hence $[x,y]\subset K\cup L$, by Lemma \ref{L1}. Therefore, there is a point $c\in[x,y]\cap K\cap L\cap H(K,u)$. In particular, $c\in F(K,u)$. Since also $c\in F(L,u)$, we have $2z-c\in F(K,-u)$. This gives 
$$z=\frac{1}{2}c+ \frac{1}{2}(2z-c) \in \frac{1}{2}F(K,u) +\frac{1}{2}F(K,-u)= Z_K(u),$$ 
as stated.

Now assume that (\ref{2.2}) holds, and suppose that $z$ is not a convexity point of $K$. Since $K\cup L$ is not convex, by Lemma \ref{L1} there exists a point $c\in {\rm bd\; conv}\,(K\cup L)\setminus (K\cup L)$. The point $c$ lies in a common support line $H(K,u)$ of $K$ and $L$, for suitable $u\in{\mathbb S}^1$. Therefore, $2z-H(K,u)$ supports $K$, hence $z\in M_K(u)$. By (\ref{2.2}), this implies that $z\in Z_K(u)$.  Moreover, $c\in [x,y]$ for suitable $x\in F(K,u)$ and $y\in F(L,u)$. By the assumption of the lemma, at least one of the sets $F(K,u)$, $F(L,u)$ is one-pointed. Suppose, first, that $F(L,u)=\{y\}$. Then the point $x':=2z-y$ satisfies $\{x'\}=F(K,-u)$. Since $z\in Z_K(u)$, there is a point $x^*\in F(K,u)$ with $z=\frac{1}{2}(x^*+x')$. This gives $y=x^*$ and hence $y\in K$, thus $c\in K$, a contradiction. Second, suppose that $F(K,u)=\{x\}$. Then the point $y':=2z-x$ satisfies  $\{y'\}=F(L,-u)$. Since $z\in Z_K(u)=2z-Z_L(u)$, we have $z\in Z_L(u)=Z_L(-u)$, hence there is a point $y^*\in F(L,u)$ with $z=\frac{1}{2}(y^*+y')$. This gives $x=y^*$ and hence $x\in L$, thus $c\in L$, again a contradiction. Thus, $z$ must be a convexity point of $K$.
\end{proof}

We show that the assumption on $K$ in Lemma \ref{L2} is not a restriction for the proof of the Theorem.

\begin{lemma}\label{L3}
If the statement of the Theorem holds under the additional assumption that $K$ has no pair of parallel edges, then it holds also without this assumption.
\end{lemma}

\begin{proof} 
Let $K\in{\mathcal K}^2$ be an arbitrary convex body. To each pair of parallel edges of $K$, there exists a $0$-symmetric segment $S$ such that one of the edges  is a translate of $S$ and the other edge contains a translate of $S$. Then $S$ is a summand of $K$ (e.g., \cite{Sch14}, Thm. 3.2.11), and there exists a convex body $C\in{\mathcal K}^2$ such that $K=C+S$, and $C$ has no pair of edges parallel to $S$. Let $S_1,S_2,\dots$ be the (finite or infinite) sequence of segments obtained in this way. Since the boundary of a planar convex body contains at most countably many segments, we can assume that the sequence $S_1,S_2,\dots$ is exhausting, that is, to each pair of parallel segments in the boundary of $K$, the shortest of the two segments is a translate of $S_i$, for suitable $i$. If there are $m$ such segments, then there is a convex body $C_m\in{\mathcal K}^2$ such that
\begin{equation}\label{1n} 
K= C_m+\sum_{i=1}^m S_i,
\end{equation}
and $C_m$ has no pair of edges parallel to one of the segments $S_1,\dots,S_m$. If there are precisely $m$ segments, we put $C_m=:C$ and $\sum_{i=1}^m S_i=:T$. If, however, the sequence $S_1,S_2,\dots$ is infinite, then the sequence $(\sum_{i=1}^m S_i)_{m\in{\mathbb N}}$ is bounded and increasing under inclusion, hence it converges (in the Hausdorff metric) to a convex body $T$, which is $0$-symmetric. From (\ref{1n}) it follows that the sequence $(C_m)_{m\in{\mathbb N}}$ converges to a convex body $C$ and that $K=C+T$. We claim that $C$ has no pair of parallel edges. In fact, suppose that there is some $u\in\Sn$ such that $F(C,u)$ and $F(C,-u)$ are segments of length at least $\ell>0$. Since $F(K,u)=F(C,u)+F(T,u)$ (see \cite{Sch14}, Thm. 1.7.5), the faces $F(K,\pm u)$ have length at least $\ell$. But then the  faces $F(T,\pm u)$ have length at least $\ell$ (by the construction of $T$), hence the faces $F(K,\pm u)$ have length at least $2\ell$. This leads to a contradiction.

If we now assume that the Theorem holds for convex bodies without a pair of parallel edges, then it holds for $C$. Thus,
for any point $z$ for which the union $C\cup(2z-C)$ is convex, the set
\begin{eqnarray*}
K\cup(2z-K) &=& (C+T)\cup (2z-C-T)=(C+T)\cup (2z-C+T)\\
&=& [C\cup (2z-C)]+T
\end{eqnarray*}
is convex. Thus, $z$ is also a convexity point of $K$. This completes the proof of Lemma \ref{L3}.
\end{proof}

For the proof of the Theorem, we consider the convex body
\begin{equation}\label{2.3} 
A_K:= {\rm conv} \bigcup_{u\in{\mathbb S}^1} Z_K(u).
\end{equation}
We shall show that each exposed point of $A_K$ is a convexity point of $K$, and that $A_K$ is two-dimensional if $K$ does not have a centre of symmetry.

\section{Proof of the Theorem}\label{sec3}

In the following, we assume that $K\in{\mathcal K}^2$ is a convex body such that the boundary of $K$ does not contain two parallel edges. As just seen, it is sufficient to prove the Theorem  for bodies satisfying this assumption.

We choose an orthormal basis $(e_1,e_2)$ of ${\mathbb R}^2$. For $\varphi\in{\mathbb R}$, we define
$$ \ur(\varphi):= (\cos\varphi)e_1+(\sin\varphi)e_2,$$
then $\ur'(\varphi)= (-\sin\varphi)e_1+(\cos\varphi)e_2$, and $(\ur(\varphi),\ur'(\varphi))$ is an orthonormal frame with the same orientation as $(e_1,e_2)$. In general, if $u\in{\mathbb S}^1$, we denote by $u'\in{\mathbb S}^1$ the unit vector such that $(u,u')$ has the  same orientation as $(e_1,e_2)$.

The support function $h_K$ of $K$ is given by $h_K(u)=\max\{\langle u,x\rangle:x\in K\}$ for $u\in{\mathbb R}^2$. We define
$$ h(\varphi):= h_K(\ur(\varphi))$$
and
\begin{equation}\label{3.0} 
p(\varphi) := \frac{1}{2}\left[h(\varphi)-h(\varphi +\pi)\right]
\end{equation}
for $\varphi\in{\mathbb R}$. Then
\begin{equation}\label{3.00}  
M_K(\ur(\varphi)) =\{x\in{\mathbb R}^2:\langle x,\ur(\varphi)\rangle =p(\varphi)\}.
\end{equation}

If a face $F(K,u)$ is one-pointed, we write $F(K,u)=\{x_K(u)\}$. For given $u\in\Sn$, at least one of the faces $F(K,u),F(K,-u)$ is one-pointed. Suppose, first, that $F(K,-u) =\{x_K(-u)\}$. The face $F(K,u)$ is a (possibly degenerate) segment, which we write as $F(K,u) = [a_K(u),b_K(u)]$, where the notation is chosen so that $b_K(u)-a_K(u)=\lambda u'$ with $\lambda\ge 0$. We set
$$ s_K(u) := \frac{1}{2}[a_K(u)+x_K(-u)],\qquad t_K(u):=\frac{1}{2}[b_K(u)+x_K(-u)].$$
If, second, $F(K,u) =\{x_K(u)\}$ is one-pointed, we write $F(K,-u) = [c_K(u),d_K(u)]$, where the notation is chosen so that $d_K(u)-c_K(u)=-\lambda u'$ with $\lambda\ge 0$, and we set
$$ s_K(u) := \frac{1}{2}[c_K(u)+x_K(u)],\qquad t_K(u):=\frac{1}{2}[d_K(u)+x_K(u)].$$
Then 
$$ Z_K(u) = [s_K(u),t_K(u)].$$
Of course, $s_K(u)=t_K(u)$ if $F(K,u)$ and $F(K,-u)$ are both one-pointed. This holds if the support function $h_K$ is differentiable at $u$ and at $-u$.

The right and left derivatives of the function $p$ at $\varphi$ are denoted by $p'_r(\varphi)$ and $p'_l(\varphi)$, respectively. They exist, since support functions have directional derivatives. 

\begin{lemma}\label{L4}
With the notations introduced above, we have
\begin{equation}\label{3.1}
p'_r(\varphi)=\langle t_K(\ur(\varphi)),\ur'(\varphi)\rangle
\end{equation}
and
\begin{equation}\label{3.2}
p'_l(\varphi)=\langle s_K(\ur(\varphi)),\ur'(\varphi)\rangle.
\end{equation}
Therefore,
\begin{equation}\label{3.3}
t_K(\ur(\varphi))= p(\varphi)\ur(\varphi)+p_r'(\varphi)\ur'(\varphi)
\end{equation}
and
\begin{equation}\label{3.4}
s_K(\ur(\varphi))= p(\varphi)\ur(\varphi)+p_l'(\varphi)\ur'(\varphi).
\end{equation}
\end{lemma}

\begin{proof} For the directional derivatives of the support function $h_K$, we see from \cite{Sch14}, Thm. 1.7.2, that, for $u\in{\mathbb S}^1$,
\begin{eqnarray*}
h_K'(u; u') &=& h_{F(K,u)}(u') =\langle b_K(u),u'\rangle,\\
h_K'(u; -u') &=& h_{F(K,u)}(-u') =\langle a_K(u),-u'\rangle.
\end{eqnarray*}
By definition,
$$ h_K'(\ur(\varphi); \ur'(\varphi)) = \lim_{\lambda\downarrow 0} \frac{h_K(\ur(\varphi)+\lambda\ur'(\varphi)) -h_K(\ur(\varphi))}{\lambda}.$$
Here,
$$ \ur(\varphi)+\lambda\ur'(\varphi) = \sqrt{1+\lambda^2}\,\ur(\varphi+\arctan\lambda)$$
and hence
\begin{eqnarray*} 
& & h_K'(\ur(\varphi); \ur'(\varphi))\\
&&=  \lim_{\lambda\downarrow 0} \frac{\sqrt{1+\lambda^2}\, h(\varphi+\arctan\lambda) -h(\varphi)}{\lambda}\\
&&= \lim_{\lambda\downarrow 0} \left[\frac{h(\varphi+\arctan\lambda) -h(\varphi)}{\arctan\lambda}\,\frac{\arctan\lambda}{\lambda} +\frac{\sqrt{1+\lambda^2}-1}{\lambda}\,h(\varphi+\arctan\lambda)\right]\\
&&= h_r'(\varphi),
\end{eqnarray*}
where $h'_r$ denotes the right derivative. Thus,
$$ h'_r(\varphi)= \langle b_K(\ur(\varphi)),\ur'\vspace{2mm}(\varphi)\rangle.$$
If $F(K,-\ur(\varphi))$ is one-pointed, we have
$$ h'(\varphi+\pi) =\langle x_K(-\ur(\varphi)),-\ur'(\varphi)\rangle.$$ 
Both equations together yield
$$ p_r'(\varphi) = \frac{1}{2}\left[h'_r(\varphi)-h'(\varphi+\pi)\right] =\frac{1}{2}\left[\langle b_K(\ur(\varphi)), \ur'(\varphi)\rangle-\langle x_K(-\ur(\varphi)),-\ur'(\varphi)\rangle\right],$$
thus
\begin{equation}\label{3.5}
p'_r(\varphi)=\langle t_K(\ur(\varphi)),\ur'(\varphi)\rangle.
\end{equation}
If $F(K,\ur(\varphi))$ is one-pointed, then
$$ h'_r(\varphi+\pi) =\langle d_K(\ur(\varphi)),-\ur'(\varphi)\rangle,\qquad h'(\varphi)=\langle x_K(\ur(\varphi)), \ur'(\varphi)\rangle,$$
which again gives (\ref{3.5}). 

For the left derivative, we obtain in a similar way that
\begin{equation}\label{}
p'_l(\varphi)=\langle s_K(\ur(\varphi)),\ur'(\varphi)\rangle.
\end{equation}

The representations (\ref{3.3}) and (\ref{3.4}) are clear from equations (\ref{3.00}), (\ref{3.1}) and (\ref{3.2}).
This completes the proof of Lemma \ref{L4}.
\end{proof}

We use Lemma \ref{L4} to prove the following characterization of centrally symmetric convex bodies in the plane.

\begin{lemma}\label{L5}
Suppose that $K\in{\mathcal K}^2$ has no pair of parallel edges. If $\dim A_K\le 1$, then $K$ is centrally symmetric.
\end{lemma}

\begin{proof} Since $\dim A_K\le 1$, all middle sets $Z_K(u)$ of $K$ lie in some line $L$, and without loss of generality we may assume that this is the line $L=\{y\in {\mathbb R}^2: \langle y,e_1\rangle =0\}$. We want to show that, in fact, $A_K$ is one-pointed.

Let $\varphi\in (0,\pi)$. The middle line $M_K(\ur(\varphi))$ intersects the line $L$ in a single point. Since all middle sets of $K$ are contained in the line $L$, the middle set $Z_K(\ur(\varphi))$ is one-pointed, hence $s_K(\ur(\varphi))=t_K(\ur(\varphi))$. By (\ref{3.1}) and (\ref{3.2}), the function $p$ is differentiable at $\varphi$, and (\ref{3.3}) gives
$$ t_K(\ur(\varphi))= p(\varphi)\ur(\varphi)+p'(\varphi)\ur'(\varphi).$$
Since $t_K(\ur(\varphi))\in L$, we have $\langle t_K(\ur(\varphi)),e_1\rangle=0$ and therefore
$$ p(\varphi)\cos\varphi -p'(\varphi)\sin\varphi=0.$$
It follows that on $(0,\pi)$ the function $p$ is of class $C^2$ and then that $(p+p'')(\varphi)=0$ for $\varphi\in (0,\pi)$. The general solution of the differential equation  $(p+p'')(\varphi)=0$  is given by $p(\varphi)=\langle c,\ur(\varphi)\rangle$ with a constant vector $c$; by continuity of $p$, the latter holds then also for $\varphi=0$. Choosing $c$ as the origin, we see from (\ref{3.0}) that $h(\varphi+\pi)=h(\varphi)$ for $[0,\pi)$, hence $K$ is centrally symmetric. 
\end{proof}

Our last lemma produces convexity points.

\begin{lemma}\label{L6} 
Suppose that $K\in{\mathcal K}^2$ has no pair of parallel edges. Then each exposed point of the convex body $A_K$ is a convexity point of $K$.
\end{lemma}

\begin{proof} Without loss of generality, we assume that $0$ is an exposed point of $A_K$ and that the orthonormal basis $(e_1,e_2)$ of ${\mathbb R}^2$ has been chosen such that 
\begin{equation}\label{3.6}
\langle x,e_2\rangle >0\quad\mbox{for each }x\in A_K\setminus\{0\}.
\end{equation} 
We denote by $L$ the line through $0$ that is spanned by $e_1$.

We intend to apply Lemma \ref{L2} to the point $0$. For that, we have to show that $0\in M_K(\ur(\varphi))$, for some $\varphi\in [0,\pi)$, can only hold if $0\in Z_K(\ur(\varphi))$.

Let $\varphi\in (-\pi/2,\pi/2)$. The middle line $M_K(\ur(\varphi))$ intersects the line $L$ in a point which we write as $f(\varphi)e_1$. The function $f$ thus defined is continuous. From (\ref{3.00}) we see that
$$ f(\varphi)=\frac{p(\varphi)}{\cos\varphi}.$$
If $f$ is differentiable at $\varphi$, this yields
\begin{equation}\label{3.7} 
f'(\varphi) = \frac{p'(\varphi)\cos\varphi+p(\varphi)\sin\varphi}{\cos^2\varphi}=\frac{\langle s_K(\ur(\varphi)), e_2\rangle} {\cos^2\varphi}.
\end{equation}

The set
$$ N:= \{\varphi\in (-\pi/2,\pi/2): f(\varphi)\not=0\}$$
is the union of open intervals $I_j$, $j\in J$, where $J$ is finite or countable. 

Since $0$ is an exposed point of ${\rm conv}\bigcup_{u\in{\mathbb S}^1} Z_K(u)$, it must be an exposed point of some middle set $Z_K(\ur(\varphi_0))$, with suitable $\varphi_0\in (-\pi/2,\pi/2)$. Let $N^c_0$ be the connected component of $(-\pi/2,\pi/2)\setminus N$ that contains $\varphi_0$. If $N^c_0=\{\varphi_0\}$, then $0\in M_K(\ur(\varphi))$ for $\varphi \in N^c_0$. If $N^c_0$ is an interval of positive length, we have $f(\varphi)=0$ for $\varphi\in N^c_0$, and we deduce from (\ref{3.7}) that $\langle s_K(\ur(\varphi)), e_2\rangle=0$ and hence from (\ref{3.6}) that $s_K(\ur(\varphi))=0$ for $\varphi\in{\rm relint}N^c_0$. It follows again that $0\in M_K(\ur(\varphi))$ for $\varphi\in N^c_0$.

Now let $j\in J$. Let $\varphi\in I_j$ be an angle such that $f$ is differentiable at $\varphi$. Then (\ref{3.7}) and (\ref{3.6}) give $f'(\varphi)>0$, since $M_K(\ur(\varphi))$ does not pass through $0$ and hence $s_K(\ur(\varphi)) \not =0$. With the exception of countably many points in $(-\pi/2,\pi/2)$, the function $h$ is differentiable at $\varphi \in(-\pi/2,\pi/2)$ and at $\varphi+\pi$, hence $p$ and thus $f$ is differentiable everywhere with countably many exceptions. We conclude that the function $f$ (which is locally Lipschitz and hence the integral of its derivative) is strictly increasing in $I_j$. But this implies that $(-\pi/2,\pi/2)\setminus N$ consists of a single closed interval, namely $N^c_0$. Therefore, no middle line $M_K(\ur(\varphi))$ with $\varphi\in(-\pi/2,\pi/2)\setminus N^c_0$ passes through $0$. The middle line $M_K(\ur(\pi/2))$ is parallel to $L$ and distinct from it and hence also does not pass through $0$. Now it follows from Lemma \ref{L2} that $0$ is a convexity point of $K$. 
\end{proof}

To complete the proof of the Theorem, we note that by Lemma \ref{L3} it suffices to prove it for a convex body $K$ without a pair of parallel edges. Assuming that $K$ is not centrally symmetric, we conclude from Lemma \ref{L5} that $A_K$ is two-dimensional. Since every convex body is the closed convex hull of its set of exposed points (e.g., \cite{Sch14}, Thm. 1.4.7), $A_K$ must have three affinely independent exposed points, and by Lemma \ref{L6}, these are convexity points of $K$.

\noindent Author's address:\\[2mm]
Rolf Schneider\\
Mathematisches Institut, Albert-Ludwigs-Universit{\"a}t\\
D-79104 Freiburg i. Br., Germany\\
E-mail: rolf.schneider@math.uni-freiburg.de

\end{document}